\documentclass[12pt, a4paper]{article}

\usepackage[english]{babel} 

\usepackage{microtype} 

\usepackage{amsmath,amsfonts,amsthm} 

\usepackage[svgnames]{xcolor} 

\usepackage[hang, small, labelfont=bf, up, textfont=it]{caption} 

\usepackage{booktabs} 

\usepackage{lastpage} 

\usepackage{graphicx} 

\usepackage{enumitem} 
\setlist{noitemsep} 

\usepackage{sectsty} 
\allsectionsfont{\usefont{OT1}{phv}{b}{n}} 

\usepackage{geometry} 

\usepackage[T1]{fontenc} 
\usepackage[utf8]{inputenc} 

\usepackage{XCharter} 

\usepackage{fancyhdr} 
\fancypagestyle{firstpage}{ 
	\fancyhf{}
}

\newcommand{\authorstyle}[1]{{\large\usefont{OT1}{phv}{b}{n}\color{DarkRed}#1}} 

\newcommand{\institution}[1]{{\footnotesize\usefont{OT1}{phv}{m}{sl}\color{Black}#1}} 

\usepackage{titling} 

\newcommand{\HorRule}{\color{DarkGoldenrod}\rule{\linewidth}{1pt}} 

\pretitle{
	\vspace{-30pt} 
	\HorRule\vspace{10pt} 
	\fontsize{32}{36}\usefont{OT1}{phv}{b}{n}\selectfont 
	\color{DarkRed} 
}

\posttitle{\par\vskip 15pt} 

\preauthor{} 

\postauthor{ 
	\vspace{10pt} 
	\par\HorRule 
	\vspace{20pt} 
}

\usepackage{lettrine} 
\usepackage{fix-cm}	

\newcommand{\initial}[1]{ 
	\lettrine[lines=3,findent=4pt,nindent=0pt]{
		\color{DarkGoldenrod}
		{#1}
	}{}%
}

\usepackage{xstring} 

\newcommand{\lettrineabstract}[1]{
	\StrLeft{#1}{1}[\firstletter] 
	\initial{\firstletter}\textbf{\StrGobbleLeft{#1}{1}} 
}

\usepackage{amssymb}

\newcommand{\C}{\mathbb C}

\newcommand{\R}{\mathbb R}

\newcommand{\Z}{\mathbb Z}

\newcommand{\rank}{\mbox{rank}}

\newcommand{\im}{\mbox{im}}

\newtheorem{Theorem}{Theorem}  
\newtheorem{Proposition}{Proposition}
\newtheorem{Corollary}{Corollary}
\newtheorem{Lemma}{Lemma}

\usepackage{amsbsy}

\title {On the Equilibrium Locus of a Parameterized Dynamical System 
with Independent First Integrals} 

\author{
	\authorstyle{Yirmeyahu Kaminski\textsuperscript{1} and Pierre Lochak\textsuperscript{2}} 
	\newline\newline 
	\textsuperscript{1}\institution{Holon Institute of Technology, Holon, Israel}\\ 
	\textsuperscript{2}\institution{Sorbonne Universit\'e, Paris, France}\\ 
}

\date{\today} 

\begin{document}

\maketitle 

\thispagestyle{firstpage} 


\lettrineabstract{}

\abstract{For a family of dynamical systems with $k > 0$ independent first integrals evolving in a compact region of an Euclidean space, we study the equilibrium locus. We show that under mild and generic conditions, it is a smooth manifold that can be viewed as the total space of a certain fiber bundle and that this bundle comes equipped with a natural connection. We then proceed to show parallel transport for this connection does exist and explore some of its properties. In particular, we elucidate how one can to some extent measure the variation of the system eigenvalues restricted to a given fiber.}


\section{Introduction}
\label{}

We study the equilibrium locus of a family of dynamical systems with $k>0$ independent first integrals; in other words we explore the zero set of a vector field in Euclidean space, with at least one first integral, giving rise to a foliation of the ambient space. More precisely the system is assumed to evolve inside a compact set $V\subset \R^n$ of maximal dimension, whose boundary $\partial V$ has the structure of a stratified manifold. 

We show that if the vector field is analytic and satisfies some genericity conditions, the equilibrium locus $E$ is a smooth manifold of dimension $m+k$, where $m$ is the number of parameters and $k$ the number of independent first integrals. Then we focus on the projection from $E$ onto the space of parameters $\Lambda$. We prove that this projection has the structure of a fiber bundle,  that we call \textit{the equilibrium bundle} , and that an Erhesmann connection can be naturally associated to it. We explicitly prove that parallel transport exists for this connection and discuss its holonomy group. Then we prove that it can be made equivalent to a Riemannian submersion provided one chooses an appropriate metric on $E$. Finally we show how one can extract some information about the variation of the (nonzero) eigenvalues  of the  Jacobian matrix of the vector field along a loop embedded in a fiber of the above projection.

The viewpoint developed in this article does not seem to appear as such in the literature. Yet, it is clearly connected with works in dynamical systems in general, more specifically differential equations. It is worth in this respect quoting papers by the Russian school. In particular, several articles by Valery Romanovski and coworkers provide a useful entry point (~\cite{ZRY}). Equilibria of systems admitting at least one first integral have been studied in particular in~\cite{Du} and related papers (see the reference list of~\cite{Du}). A more topological viewpoint is developed in~\cite{DT} and other papers by these authors (see again their reference list). It is also worth noting that the stability problem for equilibria is pervasive in the literature, at least since ... Poincaré. In the present paper it is addressed in section~\ref{sec::eigenvalues}. Finally we remark that a natural generalization of the considerations of the present paper would consist in addressing the same issues for more general trajectories, starting with the periodic ones.

The paper is organized as follows. In sections~\ref{sec-model} and~\ref{sec-nature-Lambda}, we lay down all the definitions and assumptions, and we deduce the first consequences. The other sections are the core of our work. Sections~\ref{sec::fiber_bundle} and~\ref{sec::parallel_transport} deal with the definition of \textit{the equilibrium bundle}, the natural connection on it and the construction of the parallel transport. In section~\ref{sec::riemannian_submersion}, the connection is analyzed from the perspective a Riemannian submersion. Section~\ref{sec::eigenvalues}, we discuss the variation of the non zero eigenvalues of the vector field along a loop embedded in a fiber of the equilibrium bundle.   

In section~\ref{sec-examples}, we are finally in a position to discuss examples and possible applications. The first example was investigated in~\cite{Kaminski-2018}. It is a dynamical system that describes the behavior of circular net of cells. Our work can therefore be seen as a vast generalization of this particular case. We also discuss other possible examples and applications.

\section{Definitions and first properties}
\label{sec-model}

Let us consider an $n-$dimensional $(n>0)$ smooth dynamical model $\dot{x} = f(\lambda,x)$, depending on $m$ parameters 
$\lambda$ $(m>0)$. We make the following assumptions:
\begin{enumerate}
\label{hypotheses}
\item $\lambda \in \Lambda$, where $\Lambda$ is a connected open set of $\R^m$ ($m>0$),

\item $x \in V$, where $V \subset \mathbb{R}^n$, is a compact submanifold of $\R^n$ of codimension zero, whose boundary $\partial V$ is a stratified manifold. The system is assumed to stay inside $V$, that is on $\partial V$, the vector field either points inward or is tangent to $\partial V$. 

\item \label{hyp-indep-first-int} There exist $k$  smooth  independent first integrals $h_1, \cdots, h_k$ $(k>0)$, independent of the parameter(s) $\lambda$, so that for every $i$, $h_i : \R^n \rightarrow \R$ is a smooth function.  Here the independence of the functions $h_1, \cdots, h_k$ means that their gradients are linearly independent at every points.
\end{enumerate}

Writing $M = \Lambda \times V$, for $(\lambda,x)\in M$, $f(\lambda,x)\in T_xV \simeq \R^n$, the tangent space of $V$ at $x$, so that we may and do consider $f$ as a smooth map $M \rightarrow  \R^n$ giving the components of the "velocity" at the point $x\in V$ with parameter $\lambda\in\Lambda$. We denote by $E\subset M$ the equilibrium locus defined by the vanishing 
condition $f(\lambda,x)=0$ and assume throughout that $E$ is not empty.

\begin{Proposition}
\label{prop-rank-lemma}
In the setting described by assumption~\ref{hyp-indep-first-int} above, the Jacobian matrix $\frac{\partial f}{\partial x}$ has rank at most $n-k$ on E.  
\end{Proposition}
\begin{proof}
Consider $1 \leq l \leq k$ and $\lambda \in \Lambda$. Then a solution $x(t)$ of the system satisfies $h_l(x(t))=Cst$ and taking the derivative with respect to $t$, we get: $f(\lambda,x(t))\cdot \nabla h_l(x(t))=0$ where the dot denotes the usual scalar product on $\R^n$. Since through every point $x \in V$ and any $\lambda \in \Lambda$, there is an integral curve of the vector field $f(\lambda,.)$ passing through $x$, this is valid for every $\lambda \in \Lambda$ and $x\in V$:
$$
\forall \lambda \in \Lambda, \forall x \in V, <f(\lambda,x),\nabla h_l(x) > = 0.
$$ 
Taking a further derivative, we get: 
$$
\sum_{i=0}^n \left (\frac{\partial f_i}{\partial x_j} \frac{\partial h_l}{\partial x_i} + f_i \frac{\partial^2 h_l}{\partial x_j \partial x_i} \right) = 0,
$$ 
for every $j$ ($1\le j\le n$) where we write $f$ componentwise: $f=(f_1, \cdots, f_n)$. This is equivalent to writing:

\begin{equation}
\label{eq:2nd order-order-der-H}
\left ( \frac{\partial f}{\partial x} \right)^T \nabla h_l + D^2h_l f = 0,
\end{equation}
where $D^2h_l$ denotes the Hessian matrix of $h_l$ and the superscript $T$ denotes transposition (both $\nabla h_l$ and $f$ are regarded as $n$-vectors). Now at an equilibrium point ($f=0$), the second term vanishes and thus the kernel of $\left ( \frac{\partial f}{\partial x} \right)^T$ contains the $k$ vectors $\nabla h_l(x)$. Since the first integrals $h_l$ ($l=1,\ldots, k$) are assumed to be independent, that kernel has at least dimension $k$ and $\rank \left ( \frac{\partial f}{\partial x} \right) \leq n-k$. 
\end{proof}

Observe now that the functions $h_1, \cdots, h_k$ being independent of $\lambda$, the vanishing identities  $\left ( \frac{\partial f}{\partial \lambda} \right)^T \nabla h_l = 0$ are also obtained for $l=1,\ldots,k$. We may therefore require the further:
 
\bigskip

\noindent {\bf Non-degeneracy Conditions:}\label{non-degeneracy-condition} i) The rank of the matrix $\frac{\partial f}{\partial \lambda}$ is constant equal to $\min(m,n-k)$; ii) The rank of the matrix $\frac{\partial f}{\partial x}$ is constant equal to $n-k$ for $(\lambda,x) \in E$ and $n-p(\lambda,x)$ at points $(\lambda,x) \not \in E$, where $p(\lambda,x) \leq k$; iii) For every $(\lambda,x) \in E$, $\ker \left ( \frac{\partial f}{\partial x}(\lambda,x) \right) \oplus \im \left ( \frac{\partial f}{\partial x}(\lambda,x) \right) = \R^n$.
\bigskip 

\noindent \textbf{Remarks:}
Conditions i) and ii) simply express the fact that there are no further degeneracies, beyond the one exhibited in Proposition~2.1. The third condition iii) requires that the kernel and image of the partial Jacobian matrix $\frac{\partial f}{\partial x}$ be transverse subspaces of $\R^n$, which is true for a dense open set of the matrices with a fixed a given rank. 
\medskip

\noindent \textbf{Examples:}
\begin{enumerate}
\item Consider the following vector field in $\R^2$: $f = (v(\lambda,x,y),0)$, such that for all $\lambda \in \Lambda = \R$, $v(\lambda,x,y)x \leq 0$ when $x^2+y^2=1$, so that $V$ can be taken to be the Euclidean unit ball of $\R^2$. Obviously, $h_1 = y$ is a first integral and 
$$
\frac{\partial f}{\partial x} = \left ( \begin{array}{cc}
\frac{\partial v}{\partial x}  & \frac{\partial v}{\partial y} \\ 
0 & 0
\end{array} \right).
$$
Of course, $\left(\frac{\partial f}{\partial x} \right)^T \nabla h_1 = 0$. A similar result is obtained if one consider $\frac{1}{2} y^2$ instead of $y$. The set of equilibrium points $E = \{(\lambda,x,y) | v(\lambda,x,y) = 0\}$ has dimension $2$. 

\item Consider now a vector field $f$ is $\R^3$, such that $h_1 = x^2 + y^2 + z^2$ and $h_2 = 4x^2 + 4y^2 + z^2/4$ are first integrals. These functions are obviously differentially independent, outside the union of the plane $z=0$ and the line $x=y=0$. Take $V$ to be the union of the points $\{(x,y,z)\}$ such that $h_1 = R_1$ and $h_2 = R_2$, with $R_1,R_2 \in [1,3] \times [5,15]$ and $R_2 \geq 5R_1$.

Then let us consider $f = (-\lambda y(z-x),\lambda x(z-x),0))$ with $\Lambda = \R_+^*$; it is straightforward to compute 
$$
\frac{\partial f}{\partial x}  = \left ( \begin{array}{ccc} 
\lambda y & -\lambda (z-x) & -\lambda y \\
\lambda z - 2 \lambda x & 0 &  \lambda x\\
0  & 0 & 0
\end{array}
\right) 
$$
and $E = \{(\lambda,x,y,z) \in \Lambda \times V  \vert x=z\}$. Then for $(\lambda,x,y,z) \in E$, as expected: $\left(\frac{\partial f}{\partial x} \right)^T \nabla h_1 = \left(\frac{\partial f}{\partial x} \right)^T \nabla h_2 = 0$. Also, obviously, $\dim(E) = 3$.

\end{enumerate}

\begin{Proposition}
\label{prop-Emanifold}
Consider the map $f : M=\Lambda \times V \rightarrow \R^n$ and assume the nondegeneracy 
conditions above hold true. If $f$ is an {\rm analytic} vector field, the equilibrium locus $E = f^{-1}(0)$ is a smooth real analytic variety of dimension $m+k$.

\end{Proposition}

\begin{proof}
	Since $f$ is analytic, $E$ is analytic by definition. Moreover the Jacobian matrix of $f$ has constant rank along $E$. Indeed let $J$ be this Jacobian matrix: $J = \left (\frac{\partial f}{\partial \lambda} , \frac{\partial f}{\partial x} \right)$. Then by the non-degeneracy condition, the kernel of the transpose $J^T$ is precisely the span of $\nabla h_l$ for $1 \leq l \leq k$. Therefore $\rank(J) = n-k$ along $E$ which is a smooth analytic variety of dimension $m+n - (n+k) = m + k$.  
\end{proof}

This proposition fails to be true if the field $f$ is smooth but not analytic. Indeed for such a field there may exist an open set $U \subset M$ included in $f^{-1}(0)$. In that case, the equilibrium locus will have dimension $m+n > m+k$.

\textbf{We assume from now on that $f$ is analytic and write $\boldsymbol{ \pi: E \subset \Lambda \times V \rightarrow \Lambda}$}. \textit{This is an analytic and proper map whose image we denote by $\Lambda_0\subset \Lambda$.}


\section{On the nature of $\Lambda_0$}
\label{sec-nature-Lambda}

We start with the following observation.

\begin{Lemma}
The set $\Lambda_0$ is closed in $\Lambda$. 
\end{Lemma}
\begin{proof}
This is a direct consequence of the fact that $\pi$ is proper. Alternatively in a more down-to-earth and detailed way, consider
$\lambda \in \Lambda \setminus \Lambda_0$, i.e. for every $x$ in $V$, $f(\lambda,x) \neq 0$. By continuity, there are open sets $U_x$ and $V_x$ respectively in $\Lambda$ and $V$, such that for all $(\mu,y) \in U_x \times V_x$, $f(\mu,y) \neq 0$. The collection $\{V_x\}_{x \in V}$ is an open cover of $V$. Since $V$ compact we can extract a finite subcover, $V_{x_1}, \cdots, V_{x_p}$. Writing $U = U_{x_1} \cap \cdots \cap U_{x_p}$, for all $(\mu,x) \in U \times V$, $f(\mu,x) \neq 0$ which shows that $\Lambda \setminus \Lambda_0$ is open in $\Lambda$.  
\end{proof}

\begin{Lemma}
\label{prop-fiber1}
For $\lambda \in \Lambda_0$, the fiber $E_\lambda$ of $\pi$ over $\lambda$ is a smooth real analytic $k$-dimensional manifold. 
\end{Lemma}
\begin{proof}
Given $\lambda \in \Lambda_0$ let $f_\lambda=f(\lambda, \cdot) : V\rightarrow \R^n$. By the non-degeneracy condition the rank of $f_\lambda$ is constant equal to $n-k$ over $E_\lambda=f_\lambda^{-1}(0)$ which is therefore as in the statement of the lemma. \end{proof}

\begin{Proposition}
The closed set $\Lambda_0$ has non empty interior in $\Lambda$. 
\end{Proposition}
\begin{proof}
Assume $\Lambda_0$ had empty interior and consider a point $(\lambda,x)$ where $\pi$ has a maximal rank. Then this rank is constant over a neighborhood $U$ of this point. If the rank is $m$, then $\pi$ restricted to $U$ is a submersion, so that $\pi(U)$ is open in $\Lambda$, which is impossible if $\Lambda_0$ has empty interior. 
Therefore the rank of $\pi$ is always strictly smaller than $m$ and consequently the dimension of the fibers has to be at least $k+1$, contradicting lemma~\ref{prop-fiber1} above.  
\end{proof}

From now on, we will focus on the case where $\lambda$ lies in the interior of $\Lambda_0$. Restricting attention to a connected component of $\Lambda$ if need be, we may and will assume in the sequel that $\Lambda_0$ is open and connected. Finally and for the sake of notational simplicity, we rename $\Lambda_0$ just $\Lambda$. In other words we henceforth assume that $\Lambda$ is open and connected in $\R^m$ and that the projection $\pi:E \rightarrow \Lambda$ is surjective.


\section{Fibers over the parameter space}
\label{sec::fiber_bundle}

Let us refine lemma~\ref{prop-fiber1} slightly

\begin{Proposition}
\label{prop-compact-smooth-fiber}
For every $\lambda \in \Lambda$ the fiber $E_\lambda=\pi^{-1}(\lambda)$ is a compact smooth real analytic variety of dimension $k$ whose boundary (if any) is contained in $\partial V$. 
\end{Proposition}

\begin{proof} With the notations used in the proof of lemma~\ref{prop-fiber1}, we know that $E_\lambda=f_\lambda^{-1}(0)$ is indeed a smooth real analytic variety of dimension $k$. Being a closed set in $V$, it is compact.  

Its boundary is nonempty if and only if $f_\lambda$ vanishes on the boundary $\partial V$ of $V$. Indeed if $x_0$ is a point in $E_\lambda$ {\it not} in $\partial V$, then by the implicit function theorem, $E_\lambda$ is locally the graph of a function of $k$ variables, so $x_0$ cannot lie on the boundary of $E_\lambda$. 
\end{proof}

\begin{Corollary}
If $k=1$, the fibers over $\Lambda$ have connected components which are diffeomorphic either to circles or to line segments. In the latter case, the ends of the segment lie on $\partial V$.   
\end{Corollary}

\begin{Proposition}
\label{prop-submersion}
The rank of the projection $\pi$ is constant equal to $m$. It is a surjective submersion.  
\end{Proposition}
\begin{proof}
The tangent space of $E$ at $u = (\lambda,x)$ is $\ker (J)=\ker \left ( \left [ \frac{\partial f}{\partial \lambda}, \frac{\partial f}{\partial x}   \right] \right)$. Moreover $(a,b) \in T_u(E) \cap \ker(\pi_*)$ if and only if $\frac{\partial f}{\partial \lambda} a + \frac{\partial f}{\partial x} b = 0$ and $\pi_*(a,b) = a = 0$. Hence, $\ker(\pi_{*,u}) = \{(0,b) | \frac{\partial f}{\partial x} b = 0\}$. Therefore $\dim(\ker(\pi_{*,u})) = k$ and $\rank(\pi_{*,u}) = m+k-k = m$. 
\end{proof}

Note that this result is consistent with lemma~\ref{prop-fiber1} and proposition~\ref{prop-compact-smooth-fiber}.

\begin{Proposition}
The triplet $(E,\pi,\Lambda)$ defines a fiber bundle.
\end{Proposition}

\begin{proof}
By Erhesmann's theorem it is enough to show that $\pi$ is a proper surjective submersion to conclude that it is a fiber bundle. Now it is surjective by definition, and we already know it is proper (because $V$ is compact). Finally it is a submersion by proposition~\ref{prop-submersion}. 
\end{proof}

Since $\Lambda$ is assumed to be connected, following section~\ref{sec-nature-Lambda}, we get the following corollary.

\begin{Corollary}
The fibers of $E_\lambda$ of $E$ are all diffeomorphic. 
\end{Corollary}

\noindent \textit{In the sequel the bundle $(E,\pi,\Lambda)$ will be called the equilibrium bundle}.

\section{Parallel transport of the fibers}
\label{sec::parallel_transport}

\subsection{Connection}

In this section, we aim at defining a connection on the equilibrium bundle and showing that parallel transport exists. In a very general setting, a connection on a bundle $\pi : E\rightarrow B$ is defined by an horizontal distribution, that is a subbundle of the tangent bundle $TE \rightarrow E$, say $\mathcal{H}$, such that for every $u \in E$, we have:
$$
T_u E = V_u \oplus H_u,
$$ 
where $V_u$ is the vertical summand of the tangent space, that is the set of tangent vectors $X$ at $u$, such that $\pi_*(X) = 0$. In other words $V_u\subset T_u E$ is spanned by the vectors which are tangent to the fiber over $\pi(u)$. We will denote by $\mathcal{V}$ the vertical subbundle of $TE$. 

\subsection{Transversality between fibers and integral manifolds}
\label{sec-transversality}

Given $u = (\lambda_0,x_0) \in E$, a basis of $\ker \left ( \left [\frac{\partial f}{\partial x}(u) \right]^T \right)$ is given by $( \nabla h_l(x_0))_{1 \leq l \leq k}$. Therefore the family $( \nabla h_l(x_0))_{1 \leq l \leq k}$ is also a basis of the subspace orthogonal to $\im \left ( \frac{\partial f}{\partial x}(u) \right)$, so that:
\begin{equation}
\label{eq-image-tangent-space}
\im \left ( \frac{\partial f}{\partial x}(u) \right) = \{v \in \R^{n} \mid  v^T \nabla h_l(x_0) = 0, l=1,\ldots,k\}
\end{equation}

Let $a=h(x_0) \in \R^k$ where $h=(h_1, \cdots, h_k)$ and let  $N_{a}$ be the manifold defined by $N_{a} = V \cap \left ( \cap_{l=1}^k Z(h_l-a_l)\right)$ i.e. $h_l=a_l$ on $N_a$, with $a=(a_1, \cdots, a_k)$. Equation~(\ref{eq-image-tangent-space}) simply means that:
\begin{equation}
\label{eq-image-tangent-space-explicit}
\im \left ( \frac{\partial f}{\partial x} (u) \right) = T_{x_0} N_{a}
\end{equation}

From this equation, one can state:
\begin{Proposition}
\label{prop-transversality}
Assuming the vector field $f(\lambda,x)$ satisfies the genericity condition iii) in \S 2, the manifold $N_{h(x)}$ and the fiber $E_\lambda$ are transverse at $x$ in $V$ for every $(\lambda,x) \in \Lambda\times E_\lambda$. 
\end{Proposition}

\subsection{Lifting curves}

We now show how a curve in $\Lambda$ can be naturally lifted to $E$. 
So let $\lambda(t): [0,1] \rightarrow \Lambda$ define a curve from $\lambda(0) = \lambda_1$ to $\lambda(1) = \lambda_2$ in 
$\Lambda$ (both ends are included).

Given $x \in E_{\lambda_1}$ we want to build a curve $\gamma: [0,1] \rightarrow E \cap N_{h(x)}$ starting from $x=\gamma(0)$. So  $\gamma(t)$ must satisfy:
$$
\forall t \in [0,1], \left \{ 
\begin{array}{l}
\pi(\gamma(t)) = \lambda(t) \\
\forall i, h_i(\gamma(t)) = h_i(x)
\end{array} \right .
$$

This yields the following differential system:
$$
\left \{ 
\begin{array}{l}
\frac{\partial f}{\partial \lambda} (\lambda(t),\gamma(t)) \dot{\lambda}(t) + \frac{\partial f}{\partial x} (\lambda(t),\gamma(t)) \dot{\gamma}(t) = 0 \\
\forall i, \nabla h_i(\gamma(t))^T \dot{\gamma}(t) = 0
\end{array}
\right.
$$

which can be written more compactly as:
\begin{equation}
\label{eq-curve-lifting-system}
A(t) \dot{\gamma}(t) = b(t),
\end{equation}
with 
$$
\begin{array}{cc}
A(t) = \left [  \begin{array}{c}
\frac{\partial f}{\partial x} (\lambda(t),\gamma(t)) \\
\frac{\partial h}{\partial x}(\gamma(t))
\end{array}  \right], 
& b(t) = \left [ 
\begin{array}{c}
- \frac{\partial f}{\partial \lambda}(\lambda(t), \gamma(t)) \dot{\lambda}(t) \\
0
\end{array}
\right],
\end{array}
$$
where $h = (h_1, \cdots, h_k)$. 

\begin{Proposition}
\label{prop-curve-lifting}
The differential system defined by equation~(\ref{eq-curve-lifting-system}), has a unique solution defined on the closed interval $[0,1]$. 
\end{Proposition}
\begin{proof}
By proposition~\ref{prop-transversality}, the matrix $A(t)$ has full rank $n$ for every $t \in [0,1]$. Therefore equation~(\ref{eq-curve-lifting-system}) can be written:
$$
\dot{\gamma}(t) = \left (A(t)^T A(t) \right)^{-1} A(t)^Tb(t) 
$$  
where the right-hand side actually depends on $\gamma(t)$. The functions entering on the right-hand side being smooth, {\it a fortiori} Lipschitz, elementary results about differential equations (the Cauchy-Kovaleska theorem) garantee the local existence  (in time) and uniqueness of a solution with given initial condition, thus also over a compact interval.

\end{proof}

\begin{Theorem}
\label{thm-fiber-cap-N}
For any $\lambda \in \Lambda$ and any $x \in E_\lambda$, the intersection $E_\lambda \cap N_{h(x)}$ consists in a finite set of points. Moreover the number of points in $E_\lambda \cap N_{h(x)}$ is independent of $\lambda$.
\end{Theorem}
\begin{proof}
By proposition~\ref{prop-transversality}, the intersection $E_\lambda \cap N_{h(x)}$ is discrete, hence finite since $E_\lambda$ is compact. Let $\lambda_1,\lambda_2 \in \Lambda$ and consider a curve $\lambda(t)$, defined on $[0,1]$  joining these two points. For every $x \in E_{\lambda_1}$, there is a unique lift of $\lambda(t)$ drawn on $N_{h(x)}$ and starting at $x$, which ends at some $y \in E_{\lambda_2}$. Lifts starting from a point in $E_{\lambda_2} \cap N_{h(x)}$ are also uniquely defined. This sets up a one-to-one correspondence between the points of $E_{\lambda_1} \cap N_{h(x)}$ and those of $E_{\lambda_2} \cap N_{h(x)}$.
\end{proof}

\subsection{The natural connection}
\label{sec::natural_connection}

Let $\pi':E \rightarrow \R^n$ be the projection to the second summand. Consider again the function $h: \R^n \rightarrow \R^k$, ($h = (h_1, \cdots, h_k)$) and note that $h \circ \pi'$ is a submersion. Now for $a \in \R^k$, if $(h\circ\pi')^{-1} \neq \emptyset$ it is a submanifold of $E$ of dimension $m$. Let $u = (\lambda,x) \in (h \circ \pi')^{-1}(a)$. Then the tangent space of $(h \circ \pi')^{-1}(a)$ at $u$ is $\ker \left ( [0_{k,m},\frac{\partial h}{\partial x}(x)]   \right )$. 

\begin{Lemma}
The tangent space of $(h \circ \pi')^{-1}(a)$ at $u = (\lambda,x)$ is the complement of the tangent space of $E_\lambda$ at $x$ in $T_u E$. 
\end{Lemma}
\begin{proof}
This is a direct consequence of equation~(\ref{eq-image-tangent-space-explicit}) and the non-degeneracy condition iii).
\end{proof}

This results shows that we can define a connection on $E$ choosing as the horizontal space:
$$
H_{u} = T_u((h \circ \pi')^{-1}(h(x))).
$$

We call it \textit{the natural connection of the equilibrium bundle}. 

\begin{Theorem}
The natural connection of the equilibrium bundle is flat (i.e. its curvature vanishes). 
\end{Theorem}
\begin{proof}
The curvature of a connection measures the failure of the horizontal distribution to be integrable. Here however, for $u \in E$ the subspace $H_u$ is tangent to an embedded submanifold of $E$. Therefore the distribution $(H_u)_u$ is indeed integrable and the curvature vanishes. 
\end{proof}

\begin{Theorem}
Parallel transport exists for the natural connection on the equilibrium bundle.
\end{Theorem}
\begin{proof}
It is well-known that when the fibers are compact, parallel transport exists (see~\cite[204]{Michor-2008} for details). By proposition~\ref{prop-compact-smooth-fiber}, the fibers in our case are indeed compact, so this result applies. Here in fact proposition~\ref{prop-curve-lifting} makes parallel transport explicit.

Consider two points $\lambda_1,\lambda_2 \in \Lambda$ and a curve $\lambda(t)$ joining them. Consider a point $x$ in the fiber $E_{\lambda_1}$ and the manifold $(h \circ \pi')^{-1}(h(x))$. By proposition~\ref{prop-curve-lifting}, there exists a curve $\gamma(t)$ on $E \cap (h \circ \pi')^{-1}(h(x))$ starting at $x$, such that $\forall t, \pi(\gamma(t)) = \lambda(t)$ and $\dot{\gamma(t)} \in H_{(\lambda(t),\gamma(t))}$, which explicitly shows that parallel transport exists.     
\end{proof}

The connection being flat, assuming the base manifold $\Lambda$ is simply connected, the holonomy groups vanish, which we record in the following theorem.

\begin{Theorem}
\label{thm-trivial-holonomy}
Assuming that the base manifold $\Lambda$ is simply connected, the holonomy groups of the equilibrium bundle are trivial and the fiber bundle is trivial. 
\end{Theorem}
\begin{proof}

The horizontal lift of a smooth path in $\Lambda$ entirely lies in $(h \circ \pi')^{-1}(a)$ for some $a$. Since every fiber intersects $h^{-1}(a)$ in the same finite number of points, there are finitely many horizontal lifts of a smooth loop. Therefore the holonomy groups are finite. Therefore they are Lie transformation groups of the fibers. Now since the curvature of the connection identically vanishes, the restricted holonomy groups are in fact trivial, applying a result of~\cite{Reckziegel-Wilehlmus-2006} which generalizes the Ambrose-Singer theorem to general fiber bundles.

Since the base space $\Lambda$ is assumed to simply connected, the holonomy group at a point is identical to the restricted holonomy group at this point. This leads to the conclusion.  
\end{proof}

\section{The natural connection via a Riemannian submersion}
\label{sec::riemannian_submersion}

Let us $\Phi$ be the natural connection form, $\Phi \in \Omega^1(E,TE)$ (that is for each $p \in E$ , we have $\Phi_p \in \mbox{End}(T_pE)$), which is the projection over the vertical bundle, so that for each $p \in E$, we have $\Phi_p \circ \Phi_p = \Phi_p$ and $\im(\Phi) = \mathcal{V}$, i.e. the vertical bundle. 

Let us define a Riemann metric on $E$ as follows:
$$
g(X,Y) = <\Phi(X),\Phi(Y)>_1 + <\pi_* \left(I-\Phi(X) \right), \pi_* \left(I-\Phi(Y) \right)>_2,
$$
where $<.,.>_1$ and $<.,.>_2$ respectively denote the standard inner product in $\R^m \times \R^n$ and in $\R^m$.  

When $E$ is endowed with this metric, as we shall assume from now, the projection $\pi$ becomes a Riemannian submersion. Indeed $H_u$ is the orthogonal complement of $V_u$ in $T_uE$ and the restriction $\pi_*:H_u \rightarrow T_{\pi(u)} \Lambda$ is an isometry.  

Therefore the natural connection coincides with the connection defined by $\pi$ as a Riemannian submersion. 

There is apparently no reason for which the parallel transport would define an isometry between two different fibers with the given metric. 

However, since by theorem~\ref{thm-trivial-holonomy}, the holonomy groups are trivial if we further assume that $\Lambda$ is simply connected, then in that case the parallel transport doesn't depend on the chosen curve on $\Lambda$. Therefore, we have a well defined diffeomorphism $\gamma_{\lambda_2\lambda_1}:E_{\lambda_1} \stackrel{\sim}{\rightarrow} E_{\lambda_2}$ between any two fibers, such that:
\begin{equation}
\label{eq-parallel-transp-cocycle}
\begin{array}{rcl}
\gamma_{\lambda \lambda} & = & Id_{E_{\lambda}} \\
\gamma_{\lambda_3 \lambda_1} & = & \gamma_{\lambda_3 \lambda_2} \circ \gamma_{\lambda_2 \lambda_1} 
\end{array}
\end{equation}
This observation allows transporting the metric from one fiber to another such that the parallel transport becomes an isometry between these two fibers. Choosing one fiber and transporting its metric to other fibers makes the parallel transport between any two fibers an isometry, thanks to the cocycle condition~(\ref{eq-parallel-transp-cocycle}).

\begin{Corollary}
If parallel transport is an isometry between fibers, the fibers $E_\lambda$ for all $\lambda \in \Lambda$ are totally geodesic in $E$ for the Levi-Civita connection induced by the metric on $E$. 
\end{Corollary}
\begin{proof}
Since parallel transport defines an isometry, between fibers, these latter are totally geodesic. See~\cite[proposition 1.1]{Ziller}.
\end{proof}

\section{The eigenvalues along the fibers}
\label{sec::eigenvalues}

The stability of an equilibrium point is controlled by the sign of the 
real parts of the eigenvalues of the Jacobian matrix. In this section we set up,  in the general 
framework of our discussion, a map at the level of the fundamental groups which can provide 
some information on these signs, or rather their changes.
First, note that the stability of an equilibrium is determined by the non-zero 
eigenvalues, since the system is constrained to evolve along the submanifold defined by the $k$ first integrals. 
We shall now consider the variation of the nonzero eigenvalues along a homotopy class of loops 
for a given value of the parameter.   

Consider a given value of $\lambda \in \Lambda$. It is natural to look at the map 
$e: (\lambda,x) \in E_\lambda \mapsto \mu$ where
$\mu=[\mu_1,\ldots,\mu_{n-k}]$ denotes the {\it set} of the nonzero
eigenvalues. In other words $\mu\in (\C^\star)^{n-k}/{\cal S}_{n-k}$
where for $p>0$, ${\cal S}_p$ denotes the permutation group on $p$
objects and the action here is by permutations of the factors.

Slightly more generally consider a manifold $Y$ and its $p$-fold product
$X=Y^p$ for some $p>0$. Let ${\cal S}={\cal S}_p$ act by permutation of
the factors and let $Z=X/{\cal S}$ denote the quotient. (Here and below,
for the sake of simplicity, the topological objects and morphisms are
assumed to be smooth.) One may consider $Z$ either as a topological
space, or as an orbifold. In this section we will make use of both
categories. The case at hand is relatively elementary because we are
dealing with a {\it global} quotient under the action of a {\it finite}
group. The theory of orbifolds is more general and adapted to 
the proper discontinuous action of a topological
groups. Given that in our example the application is fairly intuitive,
we refrain from recalling the details, referring instead to \cite{Thurston-78} for much more.

Given a (smooth) map $f: E\rightarrow Z$ where $E$ is a manifold, we may
consider the topological fundamental group (or functor rather) denoted
$\pi_1^{top}$ or just $\pi_1$, giving rise to a map $$\pi_1(f):
\pi_1(E)\rightarrow \pi_1(Z)$$ where $Z$ is considered as a topological
space. We may also consider the orbifold fundamental group $\pi_1^{orb}$
and the attending map $$\pi_1^{orb}(f): \pi_1^{orb}(E)\rightarrow
\pi_1^{orb}(Z).$$ But in fact $\pi_1^{orb}(E)=\pi_1^{top}(E)(=\pi_1(E))$
because $E$ is just a manifold (a ``trivial" orbifold). So we get two
maps $\pi_1(f)$ and $\pi_1^{orb}(f)$ with the same source, namely
$\pi_1(E)$, and respective targets $\pi_1(Z)$ and $\pi_1^{orb}(Z)$
according to whether $Z$ is considered in the category of topological
spaces or that of orbifolds.

There remains to compute the groups $\pi_1(Z)$ and $\pi_1^{orb}(Z)$,
which can be done easily (in our case), using \cite{Armstrong-68, Armstrong-82} 
to which we refer the reader. Indeed, assume that $G$ is a
topological group acting properly discontinuously on a connected smooth
manifold $X$. The quotient $Z=X/G$ exists as a topological space and
$\pi_1(X/G)$ is then an extension of $G/I$ by $\pi_1(X)$, where $I$ (the
letter stands for ``inertia'') is the normal subgroup of $G$ generated by
the elements whose action has fixed points in $X$. If $X=Y^p$ and
$G={\cal S}_p$ acts by permutations, it is easy to see that $I=G$ and so
$\pi_1(X/G)=\pi_1(X)=(\pi_1(Y))^p=\Z^p$ if $Y=\C^\star$.

On the other hand, $\pi_1^{orb}(X/G)$ is (essentially by definition) an
extension of $G$ itself by $\pi_1(X)$. Now if again $X=Y^p$ and $G={\cal
S}_p$ the extension splits and $\pi_1^{orb}(X/G)$ is given as a
semi-direct product $\pi_1^{orb}(X/G)=\mathcal{S}_p  \ltimes (\pi_1(Y))^p $ 
where again the action of ${\cal S}_p$ on $\pi_1(Y)^p$ is
by permutations of the factors. In our case $\pi_1(Y)=\Z$ and so
returning to the orginal situation we get a natural map
$$\pi_1^{orb}(e) : \pi_1(E_\lambda) \rightarrow \mathcal{S}_p \ltimes  \Z^p $$
which records both the monodromy of the individual (nonzero) eigenvalues
and their permutations.

What does this tell us about stability in this very general setting?
Let us  restrict attention to the topological fundamental group $\pi_1(\mu)=\Z^p$
($p=n-k$). Then for any nonzero element $(m_1, \ldots , m_p)$ in the image
of $\pi_1(e)$, with a nonzero component -- say --  $m_i$ for some index $i$, we can assert
that the eigenvalue $\mu_i$ has crossed the imaginary axis ($Re(z)=0$) at least
$2m_i$ times, so that the sign of its real part has changed at least that many times.
Although at this level of generality it seems difficult to extract more  
information from the map $e$, this may be feasible in specific, concrete cases.


\section{Examples and Applications}
\label{sec-examples}

Our first and motivating example concerns mathematical {\it biology} and consists in a circular network of cells i.e. a finite set of cells that are connected along a ring such that the last cell is connected to the first. An instance of this kind of model is the ribosome flow model on a ring, which has been introduced in~\cite{Raveh-all-2015}. 

Such a model is a parametric dynamical system defined as follows:
$$
\left \{ \begin{array}{ccl}
\dot{x}_1 & = & \lambda_n x_n(1-x_1) - \lambda_1 x_1(1-x_2) \\
\dot{x}_2 & = & \lambda_1 x_1(1-x_2) - \lambda_2 x_2(1-x_3) \\
\vdots & & \\
\dot{x}_{n-1} & = & \lambda_{n-2} x_{n-2}(1-x_{n-1}) - \lambda_{n-1} x_{n-1}(1-x_n) \\
\dot{x}_n & = & \lambda_{n-1} x_{n-1}(1-x_n) - \lambda_n x_n(1-x_1) 
\end{array} \right .
$$

The parameters $\lambda_1, \cdots, \lambda_n$ are real strictly positive numbers. They define the degree of diffusion between the cells. As shown in~\cite{Raveh-all-2015}, if the initial point lies in the hypercube $[0,1]^n$, then the system always stays within these limits. Then this system models the occupancy levels of a circular chain of $n$ sites (for example a circular DNA), while $\lambda_1, \cdots, \lambda_n$ are transition rates.   
In~\cite{Kaminski-2018}, the author analysed the structure of the equilibrium locus of this system.

\bigskip 

Coming to {\it dynamical systems} in general one is looking for systems such that  a) the trajectories are confined to a compact region of phase space b) they possess more or less natural first integrals (conserved quantities) and c) they display varieties of equlibria of positive dimensions. Requirements a) and b) are quite common and easy to satisfy : think of parametrized Hamiltonian systems and  the conservation of energy, global momentum, possibly angular momentum, or conserved quantities coming from some continuous group of geometric symmetries, via Noether's theorem. 

Requirement c) is less commonly fulfilled. The system which first comes to mind, stemming from {\it celestial mechanics}, the oldest and in some sense motivating domain in dynamical systems, is the so-called $N$-body problem, i.e. the system of differential equations which governs the motion (in ordinary three dimensional Euclidean space) of $N>1$ massive bodies interacting according to Newton's law of gravitation. This horrendously difficult problem (for $N>2$) was first studied in ``modern'' times by Lagrange and Euler, who isolated by now famous discrete sets of relative equilibria. Because the problem is so old, so natural and so difficult, literally dozens of variants of all kinds have been produced. All of them satisfy a) and b) and so does the original system. Some of them, especially those which involve blowups of certain singularities, also satisfy c) (see \cite{Devaney-1982}). A precise investigation of such systems would lead us too far afield. We hope to return to such and similar problems at some later time.

Finally we remark that it would also be interesting to explore the field of mathematical {\it economy}, where similar systems may exist in some abundance. Again such an exploration is postponed to later studies, possibly by other, more specialized authors.

\bigskip

In addition to these examples, one would consider the application of our results to control theory. Having first integrals, the system is not controllable. However lying on the manifold defined by fixed values $a = (a_1, \cdots, a_k) \in \R^k$ of the first integrals, there is a non denumerable set $E_a = E \cap \left ( \cap_{i=1}^k \pi'^{-1} \circ h_i^{-1}(a_i) \right )$  of possible equilibrium points according to the value of the parameter vector. The notations are defined in section~\ref{sec::natural_connection}. According to the discussion there, this set is actually a manifold of dimension $m$. 

By proposition~\ref{prop-curve-lifting}, every curve in $\Lambda$ can be lifted into a curve that lies in $E_a$. Therefore if the parameters are considered as commands, one can change the set of available equilibrium point(s) continuously. Then the stabilization around an equilibrium is a classical control problem that we do not recall here.

\bibliographystyle{plainnat}

\begin{thebibliography}{999}

\bibitem[Armstrong(1968)]{Armstrong-68}
M.A. Armstrong, Proc. Cambridge Philos. Soc., The fundamental group of the orbit space of a discontinuous group, 299-301, (64), 1968
	

\bibitem[Armstrong(1982)]{Armstrong-82}
M.A. Armstrong, Proceedings of the AMS, Calculating the fundamental group of an orbit space, 267-271, (84), 1982

\bibitem[Dancer-all(1982)]{DT}
Dancer, E. N.; Toland, J. F. Equilibrium states in the degree theory of periodic orbits with a first integral. Proc. London Math. Soc. (3) 63 (1991), no. 3, 569–594. (Reviewer: J. Ize) 58F22 (58C30 58E07)


\bibitem[Devaney(1982)]{Devaney-1982}
R. Devanay, Blowing Up Singularities in Classical Mechanical Systems, 
The American Mathematical Monthly, 535-552, (89), 1982


\bibitem[Dudoladov(1996)]{Du}
Dudoladov, S. L. Stability criteria for the equilibrium resonance position in systems admitting a first integral. (Russian) Regul. Khaoticheskaya Din. 1 (1996), no. 2, 77–86. (Reviewer: A. Ya. Savchenko) 34D20




\bibitem[Kaminski(2018)]{Kaminski-2018}
Y. Kaminski, Equilibrium locus of the flow on circular networks of cells, Discrete \& Continuous Dynamical Systems – S, 1169-1177, 2018

\bibitem[Michor(2008)]{Michor-2008}
Peter Michor, Topics in Differential Geometry, 2008, American Mathematical Society

\bibitem[Raveh-Zarai-Margaliot-Ruller(2015)]{Raveh-all-2015}
A. Raveh and Y. Zarai and M. Margaliot and T. Ruller, Ribosome Flow Model on a Ring, IEEE/ACM Transactions on Computational Biology and Bioinformatics, 2015


\bibitem[Reckziegel-Wilehlmus(2006)]{Reckziegel-Wilehlmus-2006}
H. Reckziegel and E. Wilehlmus, How the curvature generates the holonomy for a connection in an arbitrary fibre bundle, Results in Mathematics, 339-359, (49), 2006

\bibitem[Thurston(1978)]{Thurston-78}
W.P. Thurston, Princeton Univ. Lecture Notes, The geometry and topology of three-manifolds, 1978-1981, chap. 13


\bibitem[Zhou-all(2020)]{ZRY}
Zhou, Zhengxin; Romanovski, Valery G.; Yu, Jiang Centers and limit cycles of a generalized cubic Riccati system. Internat. J. Bifur. Chaos Appl. Sci. Engrg. 30 (2020), no. 2, 2050021, 10 pp. (Reviewer: Joan Torregrosa) 34C05 (34C07)



\bibitem[Ziller(2009)]{Ziller}
W. Ziller, Fatness Revisited, Lecture Notes, Preliminary Version,
https://www2.math.upenn.edu/~wziller/papers/Fat-09.pdf





\end{thebibliography}



\end{document}